\documentclass{{amsart}}

\usepackage[english]{babel}

\usepackage[letterpaper,top=2cm,bottom=2cm,left=3cm,right=3cm,marginparwidth=1.75cm]{geometry}

\usepackage{amsmath,amsthm,amssymb}
\usepackage{graphicx}
\graphicspath{ {./images/} }
\usepackage{color}
\usepackage{manyfoot}
\newtheorem{definition}{Definition}
\newtheorem{theorem}{Theorem}[section]
\newtheorem{corollary}[theorem]{Corollary}
\newtheorem{lemma}[theorem]{Lemma}
\newtheorem{proposition}[theorem]{Proposition}

\newcommand{\Hess}{\mbox{Hess}}

\newcommand{\be}{\begin{equation}}
\newcommand{\ee}{\end{equation}}
\newcommand{\lt}{\left}
\newcommand{\rt}{\right}

\newcommand{\goto}{\rightarrow}

\newcommand*\Lap{\mathop{}\!\mathbin\bigtriangleup}
\newcommand{\al}{\alpha}

\numberwithin{equation}{section}

\title{Super Log-concavity of  the First Eigenfunctions for Horo-convex Domains in Hyperbolic Space}
\author{Guofang Wei}
\address{
Department of Mathematics,
University of California, Santa Barbara,
CA 93106, USA
}
\thanks{GW was partially supported by NSF DMS grant  2403557}
\email{wei@math.ucsb.edu}
\author{Ling Xiao}
\address{
Department of Mathematics,  University of  Connecticut,  Storrs, Connecticut 06269-1009 USA
}
\email{ling.2.xiao@uconn.edu }
\begin{document}
\maketitle

\begin{abstract}
In this paper, we prove that the first eigenfunction of the Laplacian for a horo-convex domain $\Omega\subset\mathbb H^n$ is super log-concave when
$\text{diam}(\Omega)$ is not large. Our result is optimal in the sense that there are counterexamples 
when $\Omega$ is not horo-convex or  when
$\text{diam}(\Omega)$ is large respectively.
\end{abstract}

\section{Introduction}
\label{sec-int}
Given a bounded smooth connected domain $\Omega\subset M^n$ of a Riemannian manifold, the
eigenvalue equation of the Laplacian in $\Omega$ with the Dirichlet boundary condition is
\[\Lap\psi=-\mu\psi,\,\, \psi|_{\partial\Omega} = 0.\]
The eigenvalues $\{\mu_i| i =1, 2, \cdots\}$ satisfy $0<\mu_1<\mu_2\leq\mu_3\leq\cdots\goto\infty.$
Consider the first eigenvalue equation
\be\label{int1}
\begin{aligned}
\Lap v+\mu_1v&=0\,\,&\mbox{in $\Omega$}\\
v&=0\,\,&\mbox{on $\partial\Omega.$}
\end{aligned}
\ee
In this paper, we will study the log-concavity property of the first eigenfunction $v>0$ in  hyperbolic space.

The study of log-concavity of the first eigenfunction for the Laplacian has a long history. In 1976, Brascamp-Lieb \cite{BL76} proved that
if $v$ satisfies \eqref{int1} and $\Omega$ is a convex domain in $\mathbb R^n$, then $\log v$ is concave. Later, Caffarelli-Spruck \cite{CS82} gave an elementary proof of this result using the idea of \cite{Kor83}. When $\Omega$ is a convex domain in $\mathbb S^n,$ applying the continuity
method developed in \cite{SWYY85}, Lee-Wang proved the log-concavity of the first eigenfunction in \cite{LW87}. In \cite{KNTW}, the first author and her collaborators obtained the log-concavity of the first eigenfunction for convex domains of surfaces with positive sectional curvature.

Notice that the results mentioned above all require that the ambient manifold $M^n$ has nonnegative sectional curvature. One may ask whether the first eigenfunction of a convex domain is log-concave if not all sectional curvatures of the ambient manifold $M^n$ are nonnegative.  
Unfortunately, the answer to this question is negative even when $M^n$ is the hyperbolic space \cite{Shih1989}. In fact, it is shown in \cite{BCNSWW} that there exists a convex domain $\Omega\subset\mathbb H^n$ such that the first eigenfunction $v$ of $\Omega$ is not log-concave in any sense. That is, the largest Hessian eigenvalue of $\log v$ goes to positive infinity. This is also the case even for manifolds that have only
a single tangent plane of negative sectional curvature \cite{khan2024negative}. Hence, to prove the log-concavity of the first eigenfunction of $\Omega$ in a manifold $M^n$ that contains tangent planes of negative sectional curvature, we need more restrictions on $\Omega.$


To start, let us study the model case, that is, when $M^n$ is the hyperbolic space $\mathbb H^n.$ For $\mathbb H^n,$ a natural stronger convexity to consider is horo-convexity. Recall that a domain $\Omega\subset\mathbb H^n$ is said to be \textbf{horo-convex}, if for every point $p\in\partial\Omega$, $\Omega$ lies on the convex
side of some horosphere $S_h(p)$ through $p$. When $\Omega$ has a smooth boundary, it is equivalent to the second fundamental form of $\partial\Omega$ satisfying $\mathrm{II} \ge I_{n-1}$, that is, the principal curvatures of $\partial \Omega$ are $\ge 1$.

Indeed, we prove that the first eigenfunction of horo-convex domains is log-conacve when the diameter of the domain is not too big. In fact, we prove a stronger result. To state it, we need the following definition.

\begin{definition}
\label{def-sb1}
We say that a function $f$ is 
\textbf{$\pmb{\lambda}$-log-concave} if the matrix $\lt(-(\log f)_{ij}-\lambda|\nabla\log f|\delta_{ij}\rt)$ is positive semi-definite;
\textbf{strictly $\pmb{\lambda}$-log-concave} if the matrix $\lt(-(\log f)_{ij}-\lambda|\nabla\log f|\delta_{ij}\rt)$ is positive definite.
We also call a 1-log-concave function a \textbf{super log-concave} function.
\end{definition}

Here is our main result.
\begin{theorem}
\label{thm1}
Let $\Omega\subset\mathbb H^n$ be a horo-convex domain with diameter $\text{diam}(\Omega)\leq c_0,$ where $c_0=c_0(n)>0$ is a positive constant that only depends on $n.$
Let $v$ be the solution of the eigenvalue problem \eqref{int1}, then $v$ is super log-concave. In particular $v$ is log-concave.
\end{theorem}
 Our result is the first log-concavity estimate for the first eigenfunction of domains in a manifold with negative sectional curvatures without requiring the radial symmetry of the domain. The result is optimal in the following sense. First, as we mentioned before, there are examples that show that the first eigenfunction of convex domains in $\mathbb H^n$ does not have to be log-concave. Second, when $\Omega\subset\mathbb H^3$ is a geodesic ball
 , one can numerically show that when $\text{diam}(\Omega)>c_0$ for some $c_0\in(2, 3),$ the first eigenfunction is not super log-concave. More details will be given in Section \ref{sec-ball}.

With log-concavity of the first eigenfunction, by \cite{SWYY85}, one can reduce the estimate of the Dirichlet fundamental gap to the Neumann one. Hence combining the log-concavity of $v$ obtained in Theorem~\ref{thm1} with \cite[Theorem 3]{CLR} we get (see the proof of \cite[Corollary 2.2]{KNTW} for more details) 
\begin{corollary}
\label{cor1}
Let $\Omega\subset\mathbb H^n$ be a horo-convex domain with diameter $D:=\text{diam}(\Omega)\leq c_0$ for some $c_0=c_0(n)>0,$
$\mu_i (i=1, 2)$ be the first two eigenvalues of the Laplacian on $\Omega$ with Dirichlet boundary condition. Then
\[\mu_2-\mu_1 \ge \tfrac{\pi^2}{D^2} e^{-c(n) D}.\]
\end{corollary}

 This result improves the estimate in \cite{KST} but is weaker than the estimate in \cite{KT} when the diameter $D$ goes to zero. One can also use \cite[Theorem 14]{Bakry-Qian} to get a lower bound for $\mu_2-\mu_1$ in terms of a one-dimensional model. Note that here we only used the log-concavity of the first eigenfunction. We expect a better estimate with proper utilization of the super log-concavity.

To prove our result, the key is to prove the super log-concavity of the first eigenfunction in the sense of Definition~\ref{def-sb1}. This allows us to utilize the horo-convexity of the domain. In particular, the extra barrier function $|\nabla \log v|$ nicely contributes to controlling the ``bad'' terms caused by the negative curvature of the manifold. With this setup, we use the idea of the proof of constant rank theorems, which has been used previously by various authors (see \cite{CF85,Kor87, GM03} for examples). More precisely, we apply the continuity method as a deformation process together with the strong maximum principle to force the convexity.


The organization of this paper is as follows. In Section \ref{sec-ball}, we show that the first eigenfunction of a geodesic ball is super-log-concave when the radius of the ball is not large. In Section \ref{sec-def}, we revisit some results in \cite{HLW22}. We will use these results in Section \ref{sec-main}. Following the idea of \cite{GM03}, we prove Theorem \ref{thm1} in Section \ref{sec-main}.

Acknowledgments: This material is based on work supported by the National Science Foundation under Grant No. DMS-1928930, while the authors were in residence at the Simons Laufer Mathematical Sciences Institute
(formerly MSRI) in Berkeley, California, during the fall semester of 2024. The authors thank Ben Andrews, Julie Clutterbuck, and Malik Tuerkoen for their interest in this work.  They also thank the referee for very careful reading and comments in improving the presentation of the paper. 

\section{Super log-concavity of the first eigenfunction for balls}
\label{sec-ball}
In this section we will show that the first eigenfunction of the geodesic ball in $\mathbb H^n$ is super log-concave when the radius of the ball is not large.
Let $\mathbb S^{n-1}$ be the unit sphere in Euclidean space $\mathbb R^{n}$
with the standard induced metric $dz^2,$ then the metric of $\mathbb H^n$ is
\[g=dr^2+\sinh^2rdz^2.\]

Let $B_r\subset\mathbb H^n$ be a geodesic ball in $\mathbb H^{n}$ with radius $r.$ We denote the first eigenfunction of $B_r$ by $v_B.$ It is easy to see that $B_r$ is a strictly horo-convex domain for all $r>0.$ We will show the following.
\begin{lemma}
\label{lem-sb1}
There exists a constant $r_0=r_0(n)>0$ that depends only on $n$ such that for all $r< r_0,$ the first eigenfunction $v_B$ of $B_r\subset\mathbb H^n$ is strictly 1-log concave.
\end{lemma}
\begin{proof}
We will follow the proof of Lemma 5.2 in \cite{NSW22}. Note that by separation of variables, $v_B$ is a radial function. Denote $\varphi=(\log v_B)'$ then $\varphi$ satisfies
\be\label{sb1}
\varphi'=-\frac{n-1}{\tanh t}\varphi-\mu_1-\varphi^2,
\ee
and
\be\label{sb2}
\varphi''=\frac{n-1}{\sinh^2t}\varphi-\frac{n-1}{\tanh t}\varphi'-2\varphi\varphi'.
\ee
Here, $\mu_1$ is the first eigenvalue of $B_r.$ We already know $\varphi(0)=0$ and $\varphi<0$ on $(0, r).$ In the following, we want to show when $r<r_0,$ the matrix
\[\lt(\nabla_{ij}(\log v_B)+|\nabla\log v_B|\delta_{ij}\rt)\]
is negative definite for all $t\in(0, r)$. This is equivalent to showing that $\varphi$ satisfies
\be\label{sb3}
(\coth t)\varphi-\varphi<0
\ee
and
\be\label{sb4}
\varphi'-\varphi<0.
\ee
It is clear that for all $r>0,$ the inequality \eqref{sb3} always holds. We only need to look at inequality \eqref{sb4}.
In view of \eqref{sb1} we get
\[\varphi'(0)=-\mu_1-(n-1)\lim\limits_{t\goto 0}\frac{\varphi}{\tanh t}=-\mu_1-(n-1)\varphi'(0).\]
This gives $\varphi'(0)=-\frac{\mu_1}{n}.$

Now denote $g:=\varphi'-\varphi,$ then we have $g(0)=-\frac{\mu_1}{n}<0.$ We want to show there exists a $r_0=r_0(n)>0$ such that when
$r< r_0,$ $g(t)<0$ on $[0, r).$ If not, then for any $r>0$ there would exist some $t_1\in(0, r)$ such that
$g(t_1)=0,$ $g(t)<0$ on $[0, t_1),$ and $g'(t_1)\geq 0.$ Note that when $g=0$ we have $\varphi'=\varphi.$ Combining with \eqref{sb1} we get, at $t_1$, 
\[-(n-1)(\coth t_1)\varphi-\mu_1-\varphi^2=\varphi.\]
Denote $A:=1+(n-1)\coth t_1,$ then
\begin{equation}
    \varphi(t_1)=\frac{-A\pm\sqrt{A^2-4\mu_1}}{2}.  \label{varphi at t1}
    \end{equation}
By \eqref{sb1} and \eqref{sb2} we also obtain
\begin{align*}
g'&=\varphi''-\varphi'\\
&=\frac{n-1}{\sinh^2t}\varphi-\frac{n-1}{\tanh t}\varphi'-2\varphi\varphi'-\lt(-\frac{n-1}{\tanh t}\varphi-\mu_1-\varphi^2\rt).
\end{align*}
In particular, at $t_1$ we have
\[g'(t_1)=\frac{n-1}{\sinh^2t_1}\varphi+\mu_1-\varphi^2.\]
Therefore, $g'(t_1)\geq 0$ is equivalent to
\be\label{sb5}
\varphi^2-\frac{n-1}{\sinh^2t_1}\varphi-\mu_1\leq 0.
\ee
Denoting $B:=\frac{n-1}{2\sinh^2 t_1},$ then inequality \eqref{sb5} holds iff
\[ B-\sqrt{\mu_1+B^2} \le
 \varphi(t_1)\leq B+\sqrt{\mu_1+B^2}.
\]
Since $\varphi$ is negative, this is equivalent to
\[
|\varphi(t_1)|\leq \sqrt{\mu_1+B^2} -B=\frac{\mu_1}{B+\sqrt{B^2+\mu_1}}.
\]
Take $\varphi(t_1) = \frac{-A + \sqrt{A^2-4\mu_1}}{2}$ in \eqref{varphi at t1}, the above is equivalent to 
\[\frac{2}{A+\sqrt{A^2-4\mu_1}}\leq\frac{1}{B+\sqrt{\mu_1+B^2}}.\]

By simple algebraic computations we can see, there exists a constant $r_0=r_0(n)>0$ which satisfies the following property: For every $r< r_0$ and $t_1\in(0, r)$
we have\[\frac{2}{A+\sqrt{A^2-4\mu_1}}>\frac{1}{B+\sqrt{\mu_1+B^2}}.\]
This finishes the proof of the Lemma.
\end{proof}

\subsection{Numerical computation in $B_r\subset\mathbb H^3$}
When $n=3,$ all the eigenvalues and eigenfunctions of the balls can be explicitly computed; see, e. g. \cite[(3)]{NSW22}.  The first eigenfunction of $B_r$ is $$v_B=(\sinh t)^{-1}\sin\frac{\pi}{r}t.$$ Let $\Phi=\log v_B,$ below we give a numerically computed graph of
$\Phi''-\Phi'$ for $r=1, 2, \cdots, 9.$\footnote{Figure 1 is graphed by Matthew McGonagle using Matplotlib \cite{Hunt}.}
 This graph shows that when $r>c_0$ for some $c_0\in (2, 3),$ the first eigenfunction $v_B$ of $B_r$ is no longer super log-concave.
\begin{figure}[ht]
\centering
\includegraphics[scale=0.52]{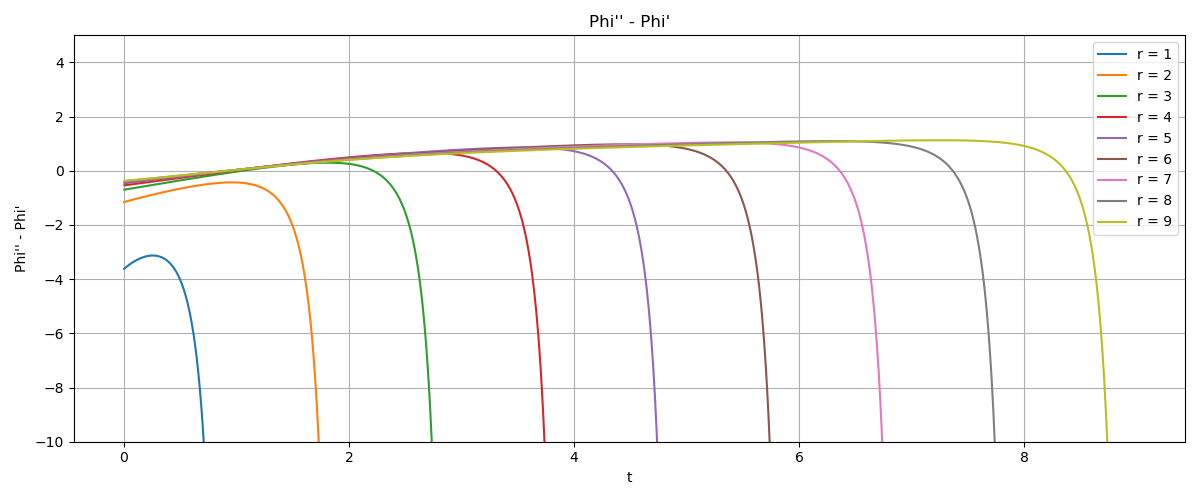}
\caption{Graph of $\Phi''-\Phi'$}
\end{figure}

\section{Deformation}
\label{sec-def}
In this section, we will provide a method that deforms a horo-convex hypersurface in the hyperbolic space into a geodesic sphere. Note that, there are many methods to do so, here we provide one that is convenient for later use.

In \cite{HLW22}, the authors studied the following flow
\be\label{sd1}
X_t=\lt[(\cosh r-u)\frac{E_{m-1}(\tilde\kappa)}{E_m(\tilde\kappa)}-u\rt]\nu,\,\,m=1, \cdots, n,
\ee
where $\nu$ is the unit outward normal of $M_t= X(M, t)$, $u = \langle \sinh r\,  \partial_r, \nu\rangle$ is the support function, $\tilde\kappa_i=\kappa_i-1$ are the shifted principal curvatures of $M_t,$ and
$$E_m(\tilde\kappa)=\frac{1}{{n\choose m}}\sum\limits_{1\leq i_1\leq\cdots\leq i_k\leq n}\tilde\kappa_{i_1}\cdots\tilde\kappa_{i_k}$$
is the normalized m-th elementary symmetric function of $\tilde\kappa.$ For each $m = 1, \cdots, n,$ let
\[\Gamma_m^+=\{x\in\mathbb R^n: E_i(x)>0, i=1, \cdots, m\}\]
be the Garding cone. They proved

\begin{theorem}[Theorem 1.6 of \cite{HLW22}]
\label{thm-sd1}
Let $X_0: M_n\goto\mathbb H^{n+1}(n\geq 2)$ be a smooth embedding such that $M_0 = X_0(M)$ is a
smooth, horo-convex hypersurface in $\mathbb H^{n+1}$ with $\tilde\kappa\in\Gamma^+_m$. The flow \eqref{sd1} then has a smooth solution
for all time $t\in[0, \infty),$ and $M_t=X_t(M)$ is strictly horo-convex for each $t > 0$ and converges smoothly
and exponentially to a geodesic sphere.
\end{theorem}
They also proved
\begin{proposition}[Proposition 6.1 of \cite{HLW22}]
\label{pro-sd1}
Let $\varphi\in C^\infty(\mathbb S^n\times[0, \infty))$ be a smooth solution to
\[\varphi_t=\lt((\cosh r-u)\frac{E_{m-1}(\tilde\kappa)}{E_m(\tilde\kappa)}-u\rt)\frac{\sqrt{1+|D\varphi|^2}}{\sinh r}.\]
Then
\[\min\limits_{\theta\in\mathbb S^n}\varphi(\cdot, 0)\leq\varphi(\cdot, t)\leq \max\limits_{\theta\in\mathbb S^n}\varphi(\cdot, 0).\]
\end{proposition}
This proposition implies that if $M_t=\{(\rho(\theta, t), \theta): (\theta, t)\in\mathbb S^n\times [0, \infty)\}$ is the family of flow hypersurfaces, then
\[\min\limits_{\theta\in\mathbb S^n}\rho(\cdot, 0)\leq\rho(\cdot, t)\leq \max\limits_{\theta\in\mathbb S^n}\rho(\cdot, 0).\]
 We will use these results in Section \ref{sec-main} to deform the horoconvex domain $\Omega$ to a geodesic ball.

\bigskip
\section{Log-concavity of the first eigenfunction}
\label{sec-main}
In this section, we will prove our main result Theorem \ref{thm1}. The idea of the proof follows \cite{GM03}.
\subsection{Set up}
\label{sub-set}
\quad\,\,Let $B_{\rho_0}\subset\mathbb H^n$ be a geodesic ball of radius $\rho_0$ in $\mathbb H^n.$ We will choose $r_0>\rho_0>0$ such that the first eigenvalue of $B_{\rho_0},$ which is denoted by $\mu_1(B),$ satisfying $\mu_1(B)\geq c^*:=c^*(n)>n^2/4.$ Here, $r_0$ is the constant determined in Lemma \ref{lem-sb1} and $c^*>n^2/4$ is a fixed constant that only depends on the dimension $n.$ It's clear that, there exists a $c_0=c_0(n)>0$ such that the above condition can be satisfied as long as $\rho_0\leq c_0.$

Now, letting $\Omega\subset B_{\rho_0}$ be a horo-convex domain containing the origin, we can flow $M_0=\partial\Omega$ by \eqref{sd1}. Then applying Theorem \ref{thm-sd1}, we know that $M_t$ is strictly horo-convex for $t>0$ and converges smoothly
and exponentially to a geometric sphere of radius $r_\infty,$ that is $\partial B_{r_\infty}.$ Moreover, by Proposition \ref{pro-sd1} we also have $M_t\subset B_{\rho_0},$ thus $r_{\infty}\leq\rho_0.$ This guarantees the first eigenfunction of $B_{r_\infty}$ is strictly 1-log-concave and $\mu_1(B_{r_\infty})\geq c^*.$ Moreover, using $\Omega_t$ to denote the domain inclosed by $M_t,$ then for all $t\in[0, \infty)$ we have $\mu_1(\Omega_t)\geq c^*.$

We want to point out that when $\Omega\subset B^2_{\rho_0}\subset\mathbb H^2$ is a 2 dimensional domain, we can always find a horo-convex domain
$\hat\Omega\subset B^3_{\rho_0}\subset\mathbb H^3$ such that $\hat\Omega\cap\{\theta_2=0\}=\Omega.$ Then we can apply Theorem \ref{thm-sd1} to
$\partial\hat\Omega$ which leads to a deformation of $\partial\Omega.$

We also want to point out that when there exists some point $p\in \partial\Omega$ such that $\tilde{\kappa}(p)\not\in\Gamma_1^+,$ that is, when
$M_0$ cannot be flowed by \eqref{sd1}, we can approximate $\Omega$ uniformly in $C^2$ by smooth horo-convex domains $\Omega^\epsilon$ satisfying $\tilde{\kappa}\in\Gamma_1^+$ for all $p\in\partial\Omega^\epsilon.$ We apply Theorem \ref{thm-sd1} to
$\partial\Omega^\epsilon$ and  obtain that the first eigenfunction of $\Omega^\epsilon$ is super log-concave. We then let $\epsilon$ tend to $0$ to complete the proof of Theorem \ref{thm1}. Therefore, in the following, without loss of generality, we will always assume that $M_0=\partial\Omega$ is a valid initial hypersurface for the flow \eqref{sd1}.

\subsection{Super log-concavity}
\label{sub-concave}
\quad\,\ In this subsection we will show that for all $t>0,$ the first eigenfunction of $\Omega_t$ is strictly 1-log-concave, that is, the matrix
\[\lt(-(\log v_{\Omega_t})_{ij}-|\nabla \log v_{\Omega_t}|\delta_{ij}\rt)\]
is positive definite, where $v_{\Omega_t}$ is the first eigenfunction of $\Omega_t.$ Since $\partial\Omega_t\goto\partial \Omega$ as $t\goto 0,$ by continuity, we will obtain the first eigenfunction of $\Omega$ is 1-log-concave, that is, super log-concave.

Recall that $\partial\Omega_t\goto\partial B_{r_\infty}$ smoothly and exponentially, by continuity we know when $t>T$ for some $T>0$ sufficiently large,  the first eigenfunction of
$\Omega_t$ is strictly 1-log-concave. Now we decrease $t$ and assume $t^*\in(0, T]$ to be the first time such that
$\lt(-(\log v_{\Omega_{t^*}})_{ij}-|\nabla \log
v_{\Omega_{t^*}}|\delta_{ij}\rt)$ is not positive definite in $\Omega_{t^*}.$ Moreover, without loss of generality, we will always assume $x_0\in\bar\Omega_{t^*}$ to be the ``most degenerate'' point in $\bar\Omega_{t^*}$, i.e., the eigenvalues of $\lt(-(\log v_{\Omega_{t^*}})_{ij}-|\nabla \log v_{\Omega_{t^*}}|\delta_{ij}\rt)$ contain the most zeros at $x_0$. In the following, we will show that $x_0$ does not exist in two steps:\\
\begin{itemize}
\item[1.] We will show that if $x_0$ exists then it's away from $\partial\Omega_{t^*}.$
\item[2.] We will show that $x_0$ cannot be an interior point of $\Omega_{t^*}.$
\end{itemize}

For simplicity, we will drop the subscript and write $v,$ $\Omega$ instead of $v_{\Omega_{t^*}},$ $\Omega_{t^*}.$
Let us denote $u:=-\log v,$ then $u$ satisfies
\be\label{sc1}
\left\{
\begin{aligned}
\Lap u&=\mu_1+w^2\,\,&\mbox{in $\Omega$}\\
u&\goto+\infty\,\,&\mbox{as $x\goto\partial\Omega,$}
\end{aligned}
\right.
\ee
where $$w=|\nabla u|.$$ We will show that the matrix
$(u_{ij}-w\delta_{ij})$ is positive definite.

\subsubsection{Strict 1-log-concavity near $\partial\Omega_t$}
\label{step1}

For convenience, we use the following definition.
\begin{definition}
\label{Def:domain-lambda-convex}
We say that a smooth domain $\Omega\subset\mathbb H^n$ is $\pmb{\lambda}$\textbf{-convex} if the second fundamental form of its boundary $\partial\Omega$ satisfies $\mathrm{II} \ge \lambda I_{n-1}$; \textbf{strictly $\pmb{\lambda}$-convex} if $\mathrm{II}>\lambda I_{n-1}.$
\end{definition}


By Theorem \ref{thm-sd1} we know that for any $t>0,$ there exists a $\lambda>1$ such that the domain $\Omega_t$ is $\lambda$-convex. In the following, we will show that the first eigenfunction of the above domain is strictly 1-log-concave near the boundary.
\begin{lemma}
 \label{sp1-lem1}
 Given a strictly $\lambda$-convex domain $\Omega \subset \mathbb H^n$ for $\lambda \in \mathbb R_+,$ let $v$ be the solution of \eqref{int1}, then $v$ is $\lambda$-log-concave in a small neighborhood of $\partial \Omega$.
\end{lemma}
We can follow the proof in \cite[Lemma 3.4]{SWW2019}.
\begin{proof}Let $\nu$ denote the unit outward normal of $\partial\Omega$. Then
  for any $p \in \partial \Omega$,  $\nabla v|_{p} = - \|\nabla v\| \nu_{p}$. Now for any $e \in T_{p}\partial\Omega$ such that $e \perp \nu$, we have
  \be
  \Hess \, v (e, e) =  \mathrm{II} (e, e)   \nabla_\nu v = - \mathrm{II} (e, e) |\nabla v|,  \label{Hess-II}
  \ee
  and
  \be
  \Hess \log v (e, e) + \lambda |\nabla \log v| \|e\|^2 = \frac 1v \left( \Hess\, v(e,e) - \tfrac{|\nabla_e v|^2}{v} + \lambda |\nabla v| \|e\|^2 \right).  \label{Hess-log}
  \ee
  Write $e = e^T +e^\perp$, where $e^\perp = \langle e, \tfrac{\nabla v}{\|\nabla v\|} \rangle \tfrac{\nabla v}{\|\nabla v\|}$. 
Then $|\nabla_e v|^2 = \langle e, \nabla v\rangle^2 = |\nabla v|^2 |e^\perp|^2$. Since as $x \to \partial \Omega$, $v \to 0$ and $|\nabla v| \not=0$, we have  $\Hess \log v (e, e) + \lambda|\nabla \log v| \|e\|^2 \to -\infty$ when 
$e^\perp \not =0$.

  For $e = e^T$, from \eqref{Hess-II} and\eqref{Hess-log},
  \be
  \label{sp1.1}
   \Hess \log v (e, e) + \lambda|\nabla \log v| \|e\|^2 = \frac 1v \left( - \mathrm{II} (e,e) |\nabla v| + \lambda|\nabla v| \|e\|^2 \right).
  \ee
 Since $\Omega$ is strictly $\lambda$-convex, we can see that \eqref{sp1.1} is negative. By continuity, we know that $v$ is $\lambda$-log-concave in a small neighborhood of $\partial \Omega$.
\end{proof}
\subsubsection{Strictly 1-log-concavity in $\Omega_t$}
\label{step2}
By the discussion from Subsection \ref{step1} we know that the degenerate point $x_0$ of the matrix $\lt(u_{ij}-w\delta_{ij}\rt)$, if it exists, must be away from
$\partial\Omega.$ In this subsection, we show that $x_0$ cannot be an interior point of $\Omega.$ In other words, $x_0$ does not exist.

\begin{lemma}
\label{lem-sp2.1}
Let $U\subset\Omega$ be a neighborhood of the degenerate point $x_0,$ and $u\in C^\infty(\Omega)$ a solution of \eqref{sc1}.
Suppose the matrix $\Lambda=(\Lambda_{ij}),$ where $\Lambda_{ij}:=u_{ij}-w\delta_{ij},$ is positive semi-definite in
$U.$ Then $\Lambda$ is positive definite in $U.$
\end{lemma}
\begin{proof}
For $x_0\in U,$ we may choose a local orthonormal frame $\{e_1, \cdots, e_n\}$ at $x_0$ so that $\Lambda$ is diagonal at $x_0,$ and $\Lambda_{ii}=\lambda_i$ for $i=1, \cdots, n.$
We may also assume
$\lambda_1\geq\lambda_2\geq\cdots\geq\lambda_n.$ Note that, by our assumption we have $\lambda_n(x_0)=0.$

First, we assume $w(x_0)=|\nabla u(x_0)|=0.$ Let $\bf{v}$ be a vector field obtained by parallel translating $e_1$ in a neighborhood $U$ of $x_0.$
Consider the matrix $\tilde\Lambda_{ij}:=u_{ij}-\lt<\nabla u, \bf{v}\rt>\delta_{ij}.$ It is clear that
$(\tilde\Lambda_{ij})$ is positive semi-definite in $U$ and $x_0$ is a degenerate point. Therefore, we have
$$\tilde\Lambda_{nn1}=0\,\,\mbox{ at $x_0.$}$$
This implies \be\label{sp2.1.1}u_{nn1}(x_0)=u_{11}(x_0).\ee On the other hand, the matrix $(u_{ij})$ is also positive semi-definite in $U.$
Hence, \be\label{sp2.1.2}u_{nn1}(x_0)=0.\ee Combining \eqref{sp2.1.1} and \eqref{sp2.1.2} we get $u_{11}(x_0)=0.$ By virtue of \eqref{sc1} this is impossible.
Therefore, at the degenerate point $x_0$ we always have $w(x_0)\neq 0.$

Next, we still use $U$ to denote a neighborhood of $x_0,$ since $w(x_0)\neq 0,$ we may also assume $w(x)\neq 0$ for any $x\in U.$
For our convenience,  we denote 
\begin{equation}
    \frac{u_i(x)}{w}=a_i(x),
    \label{a_i}
    \end{equation}
    then we have $\sum_{i=1}^na_i^2=1$.  
    
    Below, we follow the idea of the proof of Lemma 4.1 in \cite{GM03}.

 Let
\[S_k(\Lambda)= S_k(\lambda_1, \cdots, \lambda_n) = \sum\limits_{1\leq i_1\leq\cdots\leq i_k\leq n}\lambda_{i_1}\cdots\lambda_{i_k}\]
 be the $k$-th symmetric function. By our assumption that $\mu_1\geq c^*>n^2/4$, we have $$S_1(\Lambda)=\sum_{i=1}^n \lambda_i = \mu_1 + w^2 - nw = \mu_1 + (w - \tfrac n2)^2 - \frac{n^2}{4} \ge \mu_1 -\frac{n^2}{4} >0.$$ Recall that $x_0$ is the ``most degenerate'' point in $\Omega,$ thus there exists an integer $1\leq l\leq n-1$ and a positive constant $C_0>0 $ such that $S_l(\Lambda)\geq C_0$ for all $x\in U$ and $S_{l+1}(\Lambda(x_0))=0.$
 
We will denote $\phi(x):=S_{l+1}(\Lambda(x)).$
For two functions defined in an open set $U\subset\Omega,$ $x\in U,$ we say that $h(x)\lesssim k(x)$ provided there exist positive constants $c_1$ and $c_2$
such that
\[(h-k)(x)\leq c_1|\nabla\phi(x)|+c_2\phi(x).\]
We also write $h(x)\thicksim k(x)$ if $h(x)\lesssim k(x)$ and $k(x)\lesssim h(x).$ Moreover, we write $h\lesssim k$ if the above inequality holds in $U$,
with constants $c_1$ and $c_2$ depending only on $\|u\|_{C^3}, n,$ and $C_0.$

Since $S_l(\Lambda)\geq C_0>0$ for all $x\in U,$ there is a positive constant $C>0$ depending only on $\|u\|_{C^3}, n,$ and $C_0,$ such that
\[\lambda_1\geq\lambda_2\geq\cdots\geq\lambda_l\geq C.\]
Let $G=\{1, 2, \cdots, l\}$ and $B=\{l+1, \cdots, n\}$ be the ``good '' and ``bad'' sets of indices, respectively. Denote $S_k(G) = S_k(\lambda_1, \cdots, \lambda_l),$ 
as $\phi=S_{l+1}(\Lambda)$
we have
\be\label{sp2.1}
0\thicksim\phi(x)\thicksim S_l(G)\sum_{i\in B}\Lambda_{ii}\thicksim\sum_{i\in B}\Lambda_{ii}.
\ee
Recall that $\Lambda_{ii}\geq 0,$ from \eqref{sp2.1} we get when $i\in B,$ $\Lambda_{ii}(x)\thicksim 0$ for all $x\in U.$
Therefore, for $i\in B$
\be\label{sp2.2}
u_{ii}\thicksim w.
\ee
Similarly, since $\phi_\alpha=\sum_{i, j} S^{ij}_{l+1}\Lambda_{ij\alpha}$ for $S^{ij}_{l+1}=\frac{\partial S_{l+1}}{\partial\Lambda_{ij}}=S_l(\Lambda|i),$ we obtain
\be\label{sp2.2*}0\thicksim\phi_\alpha\thicksim S_l(G)\sum_{i\in B}\Lambda_{ii\alpha}\thicksim\sum_{i\in B}\Lambda_{ii\al}.\ee
Here and below we will use $(\Lambda|i)$ to denote the matrix obtained by excluding the $i$th-column and $i$th-row from $\Lambda.$
This implies
\be\label{sp2.3}
\sum_{i\in B}u_{ii\alpha}\thicksim(n-l)w_\alpha.
\ee
Finally, we will compute $\Lap\phi(x)$ in $U.$ Following the steps on page 567 of \cite{GM03}, with \cite[(4.15)]{GM03} we get
\be\label{sp2.4}
\begin{aligned}
\Lap\phi&\thicksim\sum_{\alpha, i}S^{ii}_{l+1}\Lambda_{ii\alpha\alpha}-2\sum_\alpha\sum\limits_{i\in B, j\in G}S_{l-1}(G|j)\Lambda^2_{ij\alpha}
-S_{l-1}(G)\sum_\alpha\sum\limits_{i, j\in B}\Lambda^2_{ij\alpha}\\
&\thicksim S_l(G)\sum_\alpha\sum_{i\in B}\Lambda_{ii\alpha\alpha}-2S_l(G)\sum_\alpha\sum\limits_{i\in B, j\in G}\frac{\Lambda^2_{ij\alpha}}{\lambda_j}
-S_{l-1}(G)\sum_\alpha\sum\limits_{i, j\in B}\Lambda^2_{ij\alpha}.
\end{aligned}
\ee
A straightforward calculation with the definition of $a_i$ in \eqref{a_i} and the assumption that $(u_{ij})$ is diagonalized at $x_0$ yields,
\be\label{sp2.5}
w_i=\sum_j\frac{u_ju_{ji}}{w}=a_iu_{ii},
\ee
and
\be\label{sp2.6}
\begin{aligned}
w_{ii}&=\sum_j\frac{u^2_{ji}+u_ju_{jii}}{w}-\sum_j\frac{u_ju_{ji}}{w^2}w_i\\
&=\frac{u^2_{ii}}{w}+\sum_{j\neq i}a_j(u_{iij}-u_j)+a_iu_{iii}-\frac{a^2_iu^2_{ii}}{w}.
\end{aligned}
\ee
Here, we used $u_{jii}=u_{iij}-u_j$ for $i\neq j$ in $\mathbb H^n.$
This gives,
\be\label{sp2.7}
\Lap w=\sum_i\frac{(1-a^2_i)u^2_{ii}}{w}-(n-1)w+\sum_ia_i(2ww_i).
\ee

Now for fixed $i\in B,$ applying the commutation formula $u_{ii\alpha\alpha}=u_{\al\al ii}-2u_{ii}+2u_{\al\al}$ we obtain
\be\label{sp2.8}
\begin{aligned}
\sum_\al u_{ii\al\al}&=\sum_\al(u_{\al\al ii}-2u_{ii}+2u_{\al\al})\\
&\thicksim(\Lap u)_{ii}-2nw+2\Lap u\\
&=(2w^2_i+2ww_{ii})-2nw+2(\mu_1+w^2)\\
&=2a_i^2w^2+2w\lt(\frac{u^2_{ii}}{w}+\sum_\al a_{\al}u_{ii\al}-w\rt)-2nw+2(\mu_1+w^2)\\
&=2a_i^2w^2+2w\sum_\al a_{\al}u_{ii\al}-2nw+2(\mu_1+w^2).
\end{aligned}
\ee
Here, we have used the equations \eqref{sp2.2}, \eqref{sp2.5}, and \eqref{sp2.6}.
Combining \eqref{sp2.8} with \eqref{sp2.7} we get, for any fixed $i\in B$
\be\label{sp2.9}
\sum_\al\Lambda_{ii\al\al}\thicksim 2a_i^2w^2+2w\sum_\al a_\al\Lambda_{ii\al}+2\mu_1
+2w^2-(n+1)w-\sum_\al\frac{(1-a_\al^2)u_{\al\al}^2}{w}.
\ee
Notice that the last term of \eqref{sp2.9} can be written as follows.
\begin{align*}
\sum_\al\frac{(1-a_\al^2)u_{\al\al}^2}{w}&=\sum_\al\frac{(1-a_\al^2)(\lambda^2_\al+2\lambda_\al w+w^2)}{w}\\
&=\sum_\al\frac{(1-a^2_\al)\lambda^2_\al}{w}+2\sum_\al(1-a^2_\al)\lambda_\al+\sum_\al(1-a_\al^2)w\\
&\thicksim\sum_{\al\in G}\frac{(1-a^2_\al)\lambda^2_\al}{w}+2\sum_\al\lambda_\al-2\sum_\al a^2_\al\lambda_\al+(n-1)w.
\end{align*}
In view of the equality $\sum_\al\lambda_\al=\mu_1+w^2-nw,$ equation \eqref{sp2.9} becomes
\be\label{sp2.10}
\sum_\al \Lambda_{ii\al\al}\lesssim a^2_i\lt(2w^2-\sum_{\al\in G}\frac{\lambda^2_\al}{w}\rt)+2w\sum_\al a_\al\Lambda_{ii\al}
+2\sum_\al a^2_\al\lambda_\al,\,\,i\in B,
\ee
where we have used $1-a^2_\al\geq a_i^2$ for any $\al\in G.$
Therefore, in conjunction with
\eqref{sp2.2*} we conclude,
\be\label{sp2.11}
S_l(G)\sum_\al\sum_{i\in B}\Lambda_{ii\al\al}
\lesssim S_l(G)\lt[\sum_{i\in B}a_i^2\lt(2w^2-\sum_{\al\in G}\frac{\lambda^2_\al}{w}\rt)+2(n-l)\sum_{\al\in G} a^2_\al\lambda_\al\rt].
\ee

Now, for any fixed $j\in G$ by the Cauchy-Schwarz inequality we have
\[\sum_{i\in B}\Lambda^2_{iji}\geq\frac{\lt(\sum_{i\in B}\Lambda_{iji}\rt)^2}{n-l}.\]
Using \eqref{sp2.3} we get
\begin{align*}
\sum_{i\in B}\Lambda_{iji}&=\sum_{i\in B}u_{iji}=\sum_{i\in B}(u_{iij}-u_j)\\
&\thicksim(n-l)(w_j-u_j)=(n-l)a_j\lambda_j.
\end{align*}
Therefore, we obtain for any fixed $j\in G,$
\be\label{sp2.12}
\sum_{i\in B}\frac{\Lambda^2_{iji}}{\lambda_j}\gtrsim(n-l)a_j^2\lambda_j.
\ee
Similarly, we can compute
\[\sum_{i\in B, j\in G}\frac{\Lambda^2_{ijj}}{\lambda_j}\geq\sum_{i\in B, j\in G}\frac{\Lambda^2_{ijj}}{\lambda_1}
\geq\frac{1}{\lambda_1l}\sum_{i\in B}\lt(\sum_{j\in G}\Lambda_{ijj}\rt)^2.\]
For fixed $i\in B,$ it is easy to see that
\[\sum_{j\in G}\Lambda_{ijj}=\sum_{j\in G}u_{ijj}=\sum_{j\in G}(u_{jji}-u_i).\]
In view of equation \eqref{sc1} we get
\[\sum_{j\in G}u_{jji}+\sum_{j\in B}u_{jji}=2ww_i.\]
Applying \eqref{sp2.3} we conclude for any fixed $i\in B$
\[\sum_{j\in G}u_{jji}\thicksim(2w-n+l)a_iw.\]
This gives
\[\sum_{j\in G}\Lambda_{ijj}\thicksim(2w-n)a_iw.\]
Therefore, we obtain for any fixed $i\in B,$
\be\label{sp2.13}
\sum_{j\in G}\frac{\Lambda^2_{ijj}}{\lambda_j}\gtrsim\frac{(2w-n)^2a_i^2w^2}{\lambda_1 l}.
\ee
Plugging \eqref{sp2.11}, \eqref{sp2.12}, and \eqref{sp2.13} into \eqref{sp2.4} we have
\be\label{sp2.14}
\begin{aligned}
\Lap\phi&\lesssim S_l(G)\lt[\sum_{i\in B}a^2_i\lt(2w^2-\sum_{\al\in G}\frac{\lambda^2_\al}{w}\rt)+2(n-l)\sum_{\al\in G}a^2_{\al}\lambda_\al\rt]\\
&-2S_l(G)\lt[\sum_{\al\in G}(n-l)a^2_\al\lambda_\al+\frac{(2w-n)^2}{\lambda_1l}\sum_{i\in B}a^2_iw^2\rt]\\
&=S_l(G)\sum_{i\in B}a^2_i\lt[2w^2-\sum_{\al\in G}\frac{\lambda^2_\al}{w}-2\frac{(2w-n)^2w^2}{\lambda_1l}\rt].
\end{aligned}
\ee
Since
\[\sum_{\al\in G}\lambda^2_\al>\frac{\lt(\sum_{\al\in G}\lambda_\al\rt)^2}{l}\gtrsim\frac{(\mu_1+w^2-nw)^2}{l},\]
it is easy to see that there exists $c^*=c^*(n)>0$   only depending on $n$ (e.g. $10n^2$) such that when $\mu_1\geq c^*$ we always have,
\[\lt[2w^2-\sum_{\al\in G}\frac{\lambda^2_\al}{w}-2\frac{(2w-n)^2w^2}{\lambda_1l}\rt]\leq 0.\]
By the strong maximum principle we conclude that $\phi\equiv 0$ in $U.$ Repeating the above argument, we would conclude that $\phi=0$ whenever $|\nabla u|\neq 0$ in $\Omega,$  which is impossible.
\end{proof}

From the above discussions, we know that for any $t>0,$ the first eigenfunction of $\Omega_t$ is strictly 1-log-concave. Therefore, by continuity, we know that
the first eigenfunction of $\Omega$ is 1-log-concave, that is, super log-concave. This proves Theorem \ref{thm1}.

\end{document}